\documentclass[11pt,letterpaper]{amsart}

\usepackage{mathrsfs,graphicx,latexsym,tikz,color,euscript}
\usepackage{amsfonts,amssymb,amsmath,amscd,amsthm}
\usepackage{epstopdf}
\usepackage{hyperref}
\usepackage{upgreek}
\usepackage{comment}
\usepackage{pdfsync}

\newcommand{\rr}{\mathbb R}

\newcommand{\zz}{\mathbb Z}  

\newcommand{\A}{\mathcal A}

\newcommand{\al}{\alpha}

\newcommand{\ep}{\varepsilon}

\newcommand{\dl}{\delta}

\newcommand{\sg}{\sigma}

\newcommand{\usg}{\upsigma}

\newcommand{\Lmd}{\Lambda}

\newcommand{\ey}{\frac{1}{2}}

\newcommand{\bn}{\mathbf n}

\newcommand{\xd}{\dot{x}}

\theoremstyle{plain}
\newtheorem{thm}{Theorem}[section]
\newtheorem{prop}{Proposition}[section]
\newtheorem{lem}{Lemma}[section]

\theoremstyle{definition}
\newtheorem{dfn}{Definition}[section]

\theoremstyle{remark}
\newtheorem{rem}{Remark}[section]

\begin{document}

\title[regularizable collinear periodic solutions]{regularizable collinear periodic solutions in the $n$-body problem with arbitrary masses}
 

\author{Guowei Yu}
\address{Chern Institute of Mathematics and LPMC, Nankai University, Tianjin, China}
\email{yugw@nankai.edu.cn}

\thanks{This work is supported by the National Key R\&D Program of China (2020YFA0713303), NSFC (No. 12171253), the Fundamental Research Funds for the Central Universities and Nankai Zhide Fundation.}

\begin{abstract}
For $n$-body problem with arbitrary positive masses, we prove there are regularizable collinear periodic solutions for any ordering of the masses, going from a simultaneous binary collision to another in half of a period with half of the masses moving monotonically to the right and the other half monotonically to the left. When the masses satisfy certain equality condition, the solutions have extra symmetry. This also gives a new proof of the existence of Schubart orbit, when $n=3$. 
\end{abstract}

\maketitle

\section{Introduction} 

The first solutions ever found in the three body problem were the homothetic Euler solutions discovered by Euler \cite{Euler} in 1767, where the configuration is collinear and the masses expands and shrink homothetically. Later similar collinear homothetic solutions were found by Moulton \cite{Moulton1910} in 1910 for any $n \ge 3$.



Since collisions with more than two bodies in general are not regularizable, see \cite{MG74} (except when $n$=3 with some special choices of masses, see \cite{Sie41} and \cite{Simo80}), there is no meaningful way to continue these homothetic solutions before or after the total collisions. Meanwhile simultaneous binary collisions\footnote{A simultaneous binary collision is a collision containing only binary collisions. For simplicity, the same name will be used, even if it only contains a single binary collision.} are known to be regularizable (intuitively solutions are extended through these collisions by elastic bounces, see \cite{LC20}, \cite{SL92}, \cite{EB93} and \cite{EB96}), so we may ask if there are collinear solutions going from one simultaneous binary collision to another, and if they can be extended through these collisions to obtain regularizable periodic solutions. 

The first solution of this type was found numerically by Schubart \cite{Sch1956} in 1956 for $n=3$ with two equal masses(see Remark \ref{rem; Schubart} and Figure \ref{fig:schubart}), and later proven by Moeckel \cite{Moeckel08} and Venturelli \cite{Ve08} using different methods(for cases with unequal masses, see \cite{Sh11} and \cite{Yan12}). In this paper, we will prove the following theorem which shows the existence of such solutions for any $n \ge 3$ with arbitrary choices of masses and ordering (see Figure \ref{fig:orbit} for an illuminating picture). 

\begin{thm}
\label{thm:MainThm} Let $\mathcal{S}_{\bn}$ be the permutation group of the index set $\bn = \{1, 2, \dots, n\}$ with $n \ge 3$. For any choice of positive masses, time $T>0$ and permutation $\sigma \in \mathcal{S}_{\bn}$, there is a regularizable collinear $2T$-periodic solution $x(t) = (x_i(t))_{i=1}^n \in \rr^{n}$ of the $n$-body problem with following properties.
\begin{enumerate}
\item[(a).] For any $t \in \rr$, $x(t) = x(t + 2T) = x(2T-t) \in \rr^n$. 
\item[(b).] For any $t \in (0, T)$, $x(t)$ is collision-free with $\xd_{\sg(i)}(t) >0$ for odd $i$, $\xd_{\sg(i)}(t) < 0$ for even $i$. 
\item[(c).] When $n$ is odd, $x(0)$ and $x(T)$ each contain $\ey (n-1)$ pairs of binary collision with
$$ \begin{aligned}
x_{\sg(1)}(0) < x_{\sg(2)}(0) & = x_{\sg(3)}(0) < \cdots < x_{\sg(2j)}(0) \\
& = x_{\sg(2j +1)}(0) < \cdots < x_{\sg(n-1)}(0) = x_{\sg(n)}(0);
\end{aligned}$$
$$ \begin{aligned}
x_{\sg(1)}(T) &= x_{\sg(2)}(T) < \cdots < x_{\sg(2j-1)}(T) \\
& = x_{\sg(2j)}(T) < \cdots < x_{\sg(n-2)}(T) = x_{\sg(n-1)}(T) < x_{\sg(n)}(T).
\end{aligned} $$
\item[(d).] When $n$ is even, $x(0)$  contains $\ey (n-2)$ pairs of binary collision and $x(T)$ contains $\ey n$ pairs of binary collision with
$$ \begin{aligned}
x_{\sg(1)}(0) & < x_{\sg(2)}(0) = x_{\sg(3)}(0) < \cdots < x_{\sg(2j)}(0) = x_{\sg(2j+1)}(0) \\
& < \cdots< x_{\sg(n-2)}(0)= x_{\sg(n-1)}(0) < x_{\sg(n)}(0);
\end{aligned}$$ 
$$ x_{\sg(1)}(T) = x_{\sg(2)} < \cdots < x_{\sg(2j-1)}(T) = x_{\sg(2j)}(T)< \cdots < x_{\sg(n-1)}(T) = x_{\sg(n)}(T). $$ 
\item[(e).]  $\xd_{\sg(1)}(0) = \xd_{\sg(n)}(T) =0$, when $n$ is odd, $\xd_{\sg(1)}(0) = \xd_{\sg(n)}(0) =0$, when $n$ is even. Moreover
$$ \lim_{t \to 0} \frac{m_{\sg(2j)} \xd_{\sg(2j)}(t) + m_{\sg(2j+1)} \xd_{\sg(2j+1)}(t)}{m_{\sg(2j)}+ m_{\sg(2j+1)}} = 0, \; \forall 1 \le j \le [(n-1)/2]; $$
$$ \lim_{t \to T} \frac{m_{\sg(2j-1)}\xd_{\sg(2j-1)}(t) + m_{\sg(2j)} \xd_{\sg(2j)}(t)}{m_{\sg(2j-1)} + m_{\sg(2j)}} =0, \; \forall 1 \le j \le [n/2]. $$ 

\item[(f).] $x(t)$ is $C^1$-block regularizable in the collinear $n$-body problem, and $C^0$-block regularizable in the planar or spatial $n$-body problem. 
\end{enumerate}
\end{thm}
\begin{rem}
For the definition of block regularization, see Definition \ref{dfn;sing-block-reg} and \ref{dfn;solution-blok-reg}.  
\end{rem}

\begin{figure} 
\includegraphics[width=\textwidth]{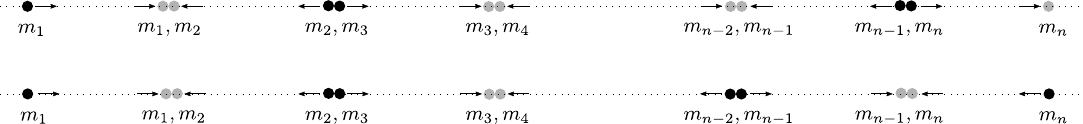}
\caption{ above $n$ is odd, below $n$ is even; $\sigma = \text{Id}$; black, $t=0$; gray, $t=T$.}
\label{fig:orbit}
\end{figure}

When the masses satisfy the following condition, we obtain solutions with extra symmetry.
\begin{equation}
\label{eq;mass-condition} m_{\sg(i)} = m_{\sg(n +1 -i)}, \; \; \forall i \in \bn.
\end{equation}

\begin{thm}
\label{thm:extra-symmetry}  If condition \eqref{eq;mass-condition} holds, for any $T>0$ and $\sg \in \mathcal{S}_{\bn}$, there is a collinear $2T$-periodic solution $x(t)=(x_i(t))_{i \in \bn} \in \rr^n$ satisfying properties given in Theorem \ref{thm:MainThm}, and for any $i \in \bn$,
\begin{enumerate}
\item[(g).] $x_{\sg(i)}(t) = -x_{\sg(n+1 -i)}(T-t)$, when $n$ is odd, $x_{\sg(i)}(t) = -x_{\sg(n+1-i)}(t)$, when $n$ is even. 
\end{enumerate}
\end{thm}
\begin{rem} \label{rem; Schubart}
When $n=3$ and $m_1 =m_3$, Theorem \ref{thm:extra-symmetry} gives an alternative proof of the Schubart orbit, as the corresponding solution going from a binary collision between $m_1$ and $m_2$ to a binary collision between $m_2$ and $m_3$ during half a period with the masses forming an Euler configuration at a quarter of a period ($x_1(T/2) = - x_3(T/2), x_2(T/2) =0$, see Figure \ref{fig:schubart}). 
\end{rem}

\begin{figure}
\centering
\includegraphics[width=\textwidth]{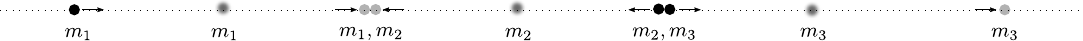}
\caption{Schubart orbit with $m_1=m_3$; black, t=0; blurred, t=T/2; gray, t=T.}
\label{fig:schubart}
\end{figure}

By Newton's universal gravitational law, the trajectories $x(t) =(x_i(t))_{i \in \bn} \in \rr^{nd}$ of $n$ point masses, $m_i>0$, $i \in \bn$, moving in $\rr^d$($d \ge 1$), satisfy the following equation
\begin{equation}
m_i \ddot{x}_i = - \sum_{j  \in \bn \setminus \{i\}} \frac{m_i m_j(x_i -x_j)}{|x_i -x_j|^3}, \;\; i \in \bn. \label{eq:n-body}
\end{equation} 
When $d=1$, it is called the collinear $n$-body problem, as all the masses are always moving in a straight line. Similarly when $d=2$ or $3$, it is called the planar or spatial $n$-body problem. Since the collinear problem is an invariant sub-system of the higher dimension problem, knowledge of it will be crucial in understanding the general problem, as shown by Moeckel in \cite{Mk07}.




Our proof is variational in nature. It is well-known that solutions of \eqref{eq:n-body} can be found as critical points of the following action functional
\begin{equation}
\label{eq;Action} \A(x; t_1, t_2) = \int_{t_1}^{t_2} L(x, \xd) \,dt = \int_{t_1}^{t_2} K(\xd) + U(x) \,dt,
\end{equation}
where 
$$ K(\xd) =  \ey \sum_{i \in \bn}  m_i |\xd_i|^2, \;\; U(x) = \sum_{ \{ i < j \} \subset \bn} \frac{m_i m_j}{|x_i - x_j|}.$$
For a long time very few results were obtained in the Newtonian $n$-body problem using variational methods, because of singularities due to collision. We denote the set of \emph{collision configurations} as
$$ \Delta = \{x = (x_i)_{i=1}^n \in \rr^{nd}: x_i = x_j, \; \text{ for some } i \ne j \}. $$

The breakthrough is the proof of the famous Figure-Eight solution in the equal mass three body problem by Chenciner and Montgomery \cite{CM00}. After that many new periodic solutions have been found and proved using action minimization methods. The main difficulty is always to show the corresponding minimizers are collision-free. It is impossible to give a complete list of the literature, we refer the readers to \cite{FT04}, \cite{Ch08},  \cite{Yu17} and the references within.

For almost all the periodic solutions found using minimization method, the masses need to satisfy various equalities required by the symmetric constraints. In particular some of the masses, if not all of them, need to be equal, for details see Section \ref{sec;extra-sym}. Up to our knowledge, the only results that hold for unequal masses are \cite{Ch08}, \cite{Chen12} and \cite{Chen21} (however the masses still need to satisfy certain conditions, which do not hold for any choice of masses). From this point of view, Theorem \ref{thm:MainThm} is the first result that establishes the existence of periodic solutions for any $n$ and any choice of masses using action minimization methods. 

The standard approaches to rule out collisions in action minimizers are \emph{local deformation}, \emph{Marchal's average method} and \emph{level estimate}. Here we use none of them, but just the monotonicity of the solution, so the proof is simple and elementary. 

As the solutions we found are collinear, they can not always be collision-free. However since they are block-regularizable, they can still be used to explain the behavior of the nearby solutions and understand the global dynamics of the problem. Compare to collision-free periodic solutions, so far there are much less results about regularizable periodic solutions based on action minimization methods. Besides \cite{Ve08}, \cite{Sh11} and \cite{Yan12} mentioned above, up to our knowledge there is also \cite{MVV13}.   




\section{Proof of Theorem \ref{thm:MainThm}}
Since the proofs for any $\sg \in \mathcal{S}_{\bn}$ are the same, we will only give the details for $\sigma =\text{Identity}$. 

To find the desired minimizer, consider the following set of admissible paths.
\begin{dfn} \label{dfn;Gamma}
 $\Gamma= \{ x \in H^1([0, T], \rr^n): x \text{ satisfies the following conditions} \}.$
\begin{enumerate}
\item[(i).] $x_1(0)x_n(0) \le 0$;
\item[(ii).] for any $t \in [0, T]$, $x_1(t) \le x_2(t) \le   \cdots \le x_{n-1}(t) \le x_n(t)$; 
\item[(iii).] when $n$ is odd, 
$$ x_1(0) \le x_2(0) = x_3(0)\le \cdots \le x_{2j}(0) = x_{2j+1}(0) \le \cdots \le x_{n-1}(0) = x_n(0), $$
$$x_1(T) = x_2(T) \le \cdots \le x_{2j-1}(T) = x_{2j}(T) \le \cdots \le x_{n-2}(T) = x_{n-1}(T) \le x_n(T); $$
\item[(iv).] when $n$ is even, 
$$  \begin{aligned}
 x_1(0) \le x_2(0) = x_3(0) & \le \cdots \le x_{2j}(0) \\
 & = x_{2j+1}(0) \le \cdots \le x_{n-2}(0) = x_{n-1}(0) \le x_n(0), 
\end{aligned}
$$
$$ x_1(T) = x_2(T) \le \cdots \le x_{2j-1}(T) = x_{2j}(T) \le \cdots \le x_{n-1}(T) = x_{n}(T). $$
\end{enumerate}
\end{dfn}

For any $x \in H^1([0, T], \rr^n)$, set $\A_T(x) = \A(x; 0, T)$.
\begin{lem} \label{lem:existence} 
There is an $x \in \Gamma$ with $\A_T(x) = \inf \{ \A_T(y): y \in \Gamma \} < \infty$. 
\end{lem}

\begin{proof}
It is not hard to see the infimum of $\A_T$ in $\Gamma$ must be finite. 

Since $\Gamma$ is weakly closed in $H^1([0, T], \rr^n)$ and $\A_T$ is weakly lower semi-continuous, by a standard result in calculus of variation, it is enough to show $\A_T(x) \to \infty$, whenever $\| x|_{[0, T]} \|_{H^1} \to \infty$. 

By Definition \ref{dfn;Gamma}, for any $x \in \Gamma$ we can define a new path $z \in H^1([0, nT], \rr^n)$ as
\begin{equation*}
z(t) = \begin{cases}
x_{2j-1}(t-2jT + 2T), &  \text{ if } t \in [(2j-2)T, (2j-1)T], \; \forall j =1, 2, \dots, [\frac{n+1}{2}] ; \\
x_{2j}(2jT-t), &   \text{ if } t \in [(2j-1)T, 2jT], \; \forall j = 1, 2, \dots, [\frac{n}{2}]. 
\end{cases}
\end{equation*}
Notice that $\|x|_{[0, T]}\|_{H^1} =\|z|_{[0, nT]}\|_{H^1}$, as 
$$ \int_0^{nT} |z|^2 \,dt = \int_0^{T} \sum_{i =1}^n |x_i|^2 \,dt, \;\;  \int_0^{nT} |\dot{z}|^2 \,dt = \int_0^{T} \sum_{i=1}^n |\xd_i|^2 \,dt. $$
Meanwhile by condition (i) in Definition \ref{dfn;Gamma}, there is a $t_0 \in [0, nT]$ satisfying $z(t_0) =0$. Then 
$$ |z(t)| = |z(t) - z(t_0)| \le \int_0^{nT} |\dot{z}| \,dt \le (nT)^{\ey} \left( \int_0^{nT} |\dot{z}|^2 \,dt \right)^{\ey}
$$
As a result, $\int_0^{nT} |z|^2 \,dt \le  n^2 T^2 \int_0^{nT} |\dot{z}|^2 \,dt$. Let $\mu= \min\{ m_1, m_2, \dots, m_n \}$. Then
$$ \begin{aligned}
\|x|_{[0, T]}\|^2_{H^1} & =\|z|_{[0, nT]}\|^2_{H^1} \le (1 + n^2T^2) \int_{0}^{nT} |\dot{z}|^2 \,dt \le (1+n^2 T^2) \int_0^{T} \sum_{i=1}^n |\xd_i|^2 \,dt \\
& \le \frac{2(1 +n^2T^2)}{\mu} \int \ey \sum_{i=1}^n m_i |\xd_i|^2 \,dt \le \frac{2(1 +n^2T^2)}{\mu}  \A_T(x). 
\end{aligned} 
$$
\end{proof}

In the rest of the section, let $x \in \Gamma$ be a minimizer of $\A_T$ in $\Gamma$. We extend $x$ to a $2T$-periodic Sobolev loop by setting
\begin{equation}
\label{eq;symmetry} x(t) = x(2T -t), \; x(t) = x(t + 2T), \; \forall t \in \rr.
\end{equation}
Then $x$ satisfies property (a) in Theorem \ref{thm:MainThm},

\begin{lem}
\label{lem;monotone} For any $0 \le t_1 < t_2 \le T$, $x_i(t_1) < x_i(t_2)$, if $i$ is odd; $x_i(t_1) > x_i(t_2)$, if $i$ is even.
\end{lem}

\begin{proof}
We give details for $i$ being odd, while the other is similar and will be omitted. 

First we prove a weaker result: for any $0 \le t_1 < t_2 \le T$, $x_i(t_1) \le x_i(t_2)$,  if $i$ is odd. By a contradiction argument, assume there is an odd $i_0$ with $x_{i_0}(t)$ strictly decreasing, when $t \in [t_1, t_2]$, for some $0 \le t_1 < t_2 \le T$. Then we can define a new path $x^* \in \Gamma$ as: 
$$ x^*_{i_0}(t) = \begin{cases}
x_{i_0}(t) & \text{ when } t \in [0, t_1], \\
2x_{i_0}(t_1) - x_{i_0}(t) & \text{ when } t \in [t_1, t_2], \\
x_{i_0}(t) + 2 (x_{i_0}(t_1) - x_{i_0}(t_2)) & \text{ when } t \in [t_2, T]; 
\end{cases}
$$
$$ x^*_{i}(t) = \begin{cases}
x_i(t) & \text{ when } t \in [0, T], \; \text{ if } i < i_0, \\
x_i(t) + 2 (x_{i_0}(t_1) - x_{i_0}(t_2)) & \text{ when } t \in [0, T], \; \text{ if } i > i_0. 
\end{cases}
$$
Notice that $\A_T(x^*) < \A_T(x)$, since  
$$ \int_{0}^T \ey \sum_{i=1}^n m_i |\dot{x}^*_i|^2 \,dt = \int_0^T \ey \sum_{i =1}^n m_i |\xd_i|^2 \,dt, \;\; \int_0^T U(x^*) \,dt < \int_0^T U(x) \,dt. $$
However this is absurd and it proves the weaker result stated above.

Now let's  assume there is an odd $i_0 < n$ and $0 \le t_1 < t_2 \le T$, such that
$$ x_{i_0}(t) = x_{i_0}(t_1), \; \forall t \in [t_1, t_2].$$
Without loss of generality, we may even assume $t_1 >0$ (if not, just replace $t_1$ by $t_2/2$). Then for any $\ep>0$ small enough, we can define a new path $\hat{x} \in \Gamma$ as follows: 
$$ \hat{x}_{i_0}(t) = \begin{cases}
x_{i_0}(t) & \text{ when } t \in [0, t_1], \\
x_{i_0}(t_1) + \frac{t-t_1}{t_2-t_1}\ep & \text{ when } t \in [t_1, t_2], \\
x_{i_0}(t) + \ep  & \text{ when } t \in [t_2, T]; 
\end{cases} $$
$$ \hat{x}_{i}(t) = \begin{cases}
x_i(t) & \text{ when } t \in [0, T], \text{ if } i < i_0, \\
x_i(t)+ \ep & \text{ when } t \in [0, T], \text{ if } i > i_0. 
\end{cases} $$
As a result, 
\begin{equation}
\label{eq;xhat-x} |\hat{x}_i(t) - \hat{x}_j(t)| \ge |x_i(t) - x_j(t)|, \; \forall t \in [0, T], \; \forall i \ne j;
\end{equation}
\begin{equation}
 \hat{x}_{n}(t) - \hat{x}_1(t) = x_{n}(t) - x_1(t) + \ep, \;\; \forall t \in [0, t_1].
 \end{equation}
Meanwhile notice that  
$$ C_1 = \sup\{ x_{n}(t) - x_1(t): t \in [0, t_1]\}>0, $$
as otherwise $U(x(t)) = \infty$, $\forall t \in [0, t_1]$. This implies $\A_T(x) =\infty$, which is absurd. Then
\begin{equation*}
\label{eq;ActionU-diff} \begin{aligned}
\int_0^T U(\hat{x}) \,dt - \int_0^{T} U(x) \,dt &  \le \int_0^{t_1} \frac{m_{1}m_n}{\hat{x}_{n}(t) - \hat{x}_1(t)} - \frac{m_1 m_n}{x_{n}(t) - x_1(t)} \,dt \\
 & = - \int_{0}^{t_1} \frac{m_{1}m_n \ep}{(x_{n}(t) - x_1(t))^2 + \ep(x_{n}(t) - x_1(t))} \,dt \\
 & \le - \frac{m_{1}m_n t_1 \ep}{C_1^2 + C_1\ep}. 
\end{aligned}
\end{equation*}
At the same time a direct computation shows 
\begin{equation*}
\label{eq:ActionK-diff} \int_0^T K(\dot{\hat{x}}) \,dt - \int_0^T K(\xd)\,dt = \frac{m_{i_0}}{2} \int_{t_1}^{t_2} \dot{\hat{x}}_{i_0}^2 - \xd_{i_0}^2 \,dt = \frac{m_{i_0}}{2 (t_2-t_1)} \ep^2.
\end{equation*}
Then the following inequality holds, for $\ep \in (0, C_1)$ small enough, which is absurd. 
$$ \A_T(\hat{x})- \A_T(x) \le - \frac{m_{1}m_n t_1}{2 C_1^2} \ep + \frac{m_{i_0}}{2(t_2 -t_1)} \ep^2 <0.$$ 

Now the only case that is not covered by the above argument is when $i_0 =n$ with $n$ being odd. In this case, assume $t_2 < T$ (if not, just replace $t_2$ by $(t_1 +t_2)/2$) and like above, we have 
$$ C_2 = \sup \{ x_n(t) - x_1(t): t \in [t_2, T]\}>0.$$
Then the same $\hat{x}$ defined as above satisfies \eqref{eq;xhat-x} and 
$$ \hat{x}_n(t) - \hat{x}_1(t) = x_n(t) - x_1(t) + \ep, \;\; \forall t \in [t_1, t_2].$$
A contradiction now can be reached by a similar estimate as above. 
\end{proof}

\begin{prop}
\label{prop;CollFree-Solution}  $x(t)$, $t \in (0, T)$,  is a collision-free smooth solution of \eqref{eq:n-body}. 
\end{prop}

By Lemma \ref{lem;monotone} $x(t)$ is collision-free, $\forall t\in (0, T)$. As a result, if we make small enough perturbations near $x(t)$, the new path still belongs to $\Gamma$. This means that $x|_{[t -\dl, t+\dl]}$ is a fix-end local minimizer of $\A$, for $\dl>0$ small enough, which implies the above proposition.



\begin{proof}[Proof of (b), (c) \& (d) of Theorem \ref{thm:MainThm}]

For property (b), let's assume $i$ is odd (the proof for even $i$ is similar). By a contradiction argument, assume there is an odd $i_0 <n$ and a $t_0 \in (0, T)$ with $\xd_{i_0}(t_0)=0$. By Lemma \ref{lem;monotone} and Proposition \ref{prop;CollFree-Solution}, there is a constant $C_1>0$, such that for any $\ep>0$ small enough,
\begin{equation}
\label{eq;x-dot-upper-bound} 0 \le \dot{x}_{i_0}(t) \le C_1(t-t_0), \; \forall t \in [t_0, t_0 +\ep]. 
\end{equation}
Now define a new path $\hat{x} \in H^1([0, T], \rr^n)$ as follows: 
$$ \hat{x}_{i_0}(t) = \begin{cases}
x_{i_0}(t) & \text{ when } t \in [0, t_0], \\
x_{i_0}(t) + (t-t_0)(2\ep-(t-t_0)) & \text{ when } t \in [t_0, t_0 +\ep], \\
x_{i_0}(t) + \ep^2 & \text{ when } t \in [t_0 + \ep, T];
\end{cases}
$$
$$ \hat{x}_i(t) = \begin{cases}
x_i(t) & \text{ when } t \in [0, T] \text{ and } i < i_0, \\
x_i(t) + \ep^2 & \text{ when } t \in [0, T] \text{ and } i > i_0. 
\end{cases} $$
As a result, 
$$ |\hat{x}_i(t) - \hat{x}_j(t)| \ge |x_i(t) - x_j(t)|, \; \forall t \in [0, T], \; \forall i \ne j; $$ 
$$ \hat{x}_{n}(t) - \hat{x}_1(t) = x_{n}(t) - x_1(t) + \ep^2, \;\; \forall t \in [0, t_0]. $$
Meanwhile as $C_2 = \sup \{ x_n(t) - x_1 (t): t \in [0, t_0] \}$ is a positive constant, we have
$$ \begin{aligned}
\int_0^T U(\hat{x}) \,dt - \int_0^{T} U(x) \,dt &  \le \int_0^{t_0} \frac{m_{1}m_n}{\hat{x}_{n}(t) - \hat{x}_1(t)} - \frac{m_1 m_n}{x_{n}(t) - x_1(t)} \,dt \\
 & \le - \int_{0}^{t_0} \frac{m_{1}m_n \ep^2}{(x_{n}(t) - x_1(t))^2 + \ep^2(x_{n}(t) - x_1(t))} \,dt \\
 & \le - \frac{m_{1}m_n t_0 \ep^2}{C_2^2 + C_2\ep^2}. 
\end{aligned}
$$
The definition of $\hat{x}$ and \eqref{eq;x-dot-upper-bound} then imply 
$$ \begin{aligned}
\int_0^T K(\dot{\hat{x}}) - K(\xd) \,dt & = \frac{m_{i_0}}{2} \int_{t_0}^{t_0 +\ep} \dot{\hat{x}}^2_{i_0} - \xd^2_{i_0} \,dt \\
& = 2 m_{i_0} \int_{t_0}^{t_0 +\ep} \dot{x}_{i_0}(t) (\ep+t_0 -t) + (\ep + t_0 -t)^2 \,dt  \\
& \le \int_{t_0}^{t_0 +\ep} C_1(t-t_0)(\ep +t_0 -t) + (\ep + t_0 -t)^2 \,dt  \\
& \le \frac{(2 + C_1) m_{i_0}}{3}\ep^3. 
\end{aligned}
$$
By the above estimates, for $  \ep \in (0, \sqrt{C_2})$ small enough, we have the following, which is absurd.  
$$ \A_T(\hat{x}) - \A_T(x) \le -\frac{m_1 m_n t_0}{2 C_2^2}\ep^2 +  \frac{(2 + C_1) m_{i_0}}{3}\ep^3 <0.$$ 

What's left now is when $i_0 =n$ (with $n$ being odd) and $\dot{x}_{i_0}(t_0) = 0$, for some $t_0 \in (0, T)$. In this case, a contradiction can be reached by a similar argument as above, as
$$ \hat{x}_n(t) - \hat{x}_1(t) = x_n(t) - x_1(t) +\ep^2, \; \forall t \in [t_0, T], $$
and $C_3 = \sup \{ x_n(t) -x_1(t): t \in [t_0, T] \}$ is a positive constant $C_3$. 

For property (c) and (d), as $x \in \Gamma$, it must satisfy Condition (iii) and (iv) in Definition \ref{dfn;Gamma}. Then Lemma \ref{lem;monotone} implies the inequalities in these conditions must be strict.
\end{proof}

Recall that the velocity of the center of mass of a smooth solutions must be a constant. 
\begin{lem}
\label{prop:center-of-mass} $v_0  = \frac{\sum_{i \in \bn} m_i \xd_i(t)}{m_0}=0, \forall t \in (0, T)$, where $m_0 = \sum_{i \in \bn} m_i$. 
\end{lem}

\begin{proof}
Notice the path $\hat{x}|_{[0, T]}$ with $\hat{x}_i(t) = x_i(t)- v_0 t$, $\forall i \in \bn$, still belongs to $\Gamma$, and $\forall t \in [0, t]$, 
$$ 2K(\xd) - 2 K(\dot{\hat{x}}) = \sum_{i \in \bn} m_i \big(\xd_i^2 - (\xd_i -v_0)^2 \big) = \sum_{i \in \bn} m_i\big( 2v_0 \xd_i - v_0^2 \big) = m_0 v_0^2.$$
As $U(\hat{x}(t)) = U(x(t))$, $v_0$ must be zero. Otherwise the following inequality holds, which is absurd.
$$\A_T(x) - \A_T(\hat{x}) = \int_0^{T} K(\xd) - K(\dot{\hat{x}}) \,dt = \ey m_0v_0^2T  >0.$$ 
\end{proof}

\begin{proof}[Proof of (e) of Theorem \ref{thm:MainThm}] By the definition of $\Gamma$, for any $\xi \in C^{1}([0, T], \rr^n) \cap \Gamma$, $x+ \ep \xi \in \Gamma$, when $\ep>0$ is small enough. In the following it is always assumed that any $\xi$ that we choose from $C^{1}([0, T], \rr^n) \cap \Gamma$ satisfying
\begin{equation}
\label{eq;action-xi-finite} \A_T(x +\ep \xi) <\infty, \; \text{ when } \ep>0 \text{ is small enough}. 
\end{equation}
Then the minimizing property of $x(t)$, Sundman estimates and Lebesgue dominated convergence theorem imply 
$$ \frac{d}{d \ep} \Big|_{\ep =0} \A_T(x + \ep \xi) = \int_0^{T} \sum_{i \in \bn} (m_i \xd_i \dot{\xi}_i + \partial_{x_i}U(x)\xi_i ) \,dt =0. $$
By \eqref{eq:n-body} and integration by parts,  
\begin{equation}
\label{eq;xi-bdry-old}  \lim_{t \to T^-} \sum_{i \in \bn} m_i \xd_i(t) \xi_i(t) - \lim_{t \to 0^+} \sum_{i \in \bn} m_i \xd_i(t) \xi_i(t) =0. 
\end{equation}

It is known during a binary collision, the velocity of the center of mass of the pair involved in the collision has a finite limit, as they approaches the binary collision (see \cite{Sp70} or \cite[Section 4]{FT04}). This implies the following one-sided limits are well-defined and finite, 
$$ \eta^+_j = \lim_{t \to 0^+} m_{2j} \xd_{2j}(t) + m_{2j+1} \xd_{2j+1}(t), \; \forall 1 \le j \le [(n-1)/2];$$
$$ \eta^-_j = \lim_{t \to T^-} m_{2j-1} \xd_{2j-1}(t) + m_{2j} \xd_{2j}(t), \; \forall 1 \le j \le [n/2].$$
We shall use these to rewrite \eqref{eq;xi-bdry-old}. 

First let's consider the case that $n$ is odd. Notice that for any $j$,
$$ \begin{aligned}
 m_{2j} \dot{x}_{2j} \xi_{2j}  + m_{2j+1} \dot{x}_{2j+1} \xi_{2j+1} & =  (m_{2j} \dot{x}_{2j} + m_{2j+1} \dot{x}_{2j+1}) \frac{\xi_{2j} + \xi_{2j+1}}{2} \\ 
 &+ (m_{2j} \dot{x}_{2j} - m_{2j+1} \dot{x}_{2j+1}) \frac{\xi_{2j} - \xi_{2j+1}}{2}.
\end{aligned} $$
By Sundman's asymptotic estimate (see \cite{Sp70} or \cite[Section 4]{FT04})
$$ \lim_{t \to 0^+} \big(m_{2j} \dot{x}_{2j}(t) - m_{2j+1} \dot{x}_{2j+1}(t) \big) \big(\xi_{2j}(t) - \xi_{2j+1}(t) \big) = 0. $$
As a result,  
\begin{equation} \label{eq;lim-0-odd}
\lim_{t \to 0^+} \sum_{i \in \bn} m_i \xd_i(t) \xi_i(t) =  
m_1 \dot{x}_1(0) \xi_1(0) + \sum_{j =1}^{[\frac{n-1}{2}]}  \eta_j^+ \frac{\xi_{2j}(0)+ \xi_{2j+1}(0)}{2}. 
\end{equation}
A similar result can be obtained, as $t$ goes to $T^-$,
\begin{equation} \label{eq;lim-T-odd}
\lim_{t \to T^-} \sum_{i \in \bn} m_i \xd_i(t) \xi_i(t) =  
  \sum_{j =1}^{[\frac{n}{2}]}  \eta_j^- \frac{\xi_{2j-1}(T)+ \xi_{2j}(T)}{2} + m_n \dot{x}_n(T) \xi_n(T).
\end{equation}
With \eqref{eq;lim-0-odd} and \eqref{eq;lim-T-odd}, when $n$ is odd, we can rewrite \eqref{eq;xi-bdry-old} as
\begin{equation}
\label{eq;lim-0-T-odd} \begin{aligned}
0 & = \left( m_1 \dot{x}_1(0) \xi_1(0) + \sum_{j =1}^{[n-1/2]}  \eta_j^+ \frac{\xi_{2j}(0)+ \xi_{2j+1}(0)}{2} \right) \\ 
& - \left( \sum_{j =1}^{[n/2]}  \eta_j^- \frac{\xi_{2j-1}(T)+ \xi_{2j}(T)}{2} + m_n \dot{x}_n(T) \xi_n(T) \right).
\end{aligned}
\end{equation}
When $n$ is even, by a similar argument, we can rewrite \eqref{eq;xi-bdry-old} as
\begin{equation}
\label{eq;lim-0-T-even} \begin{aligned}
0 & = \left( m_1 \dot{x}_1(0) \xi_1(0) + \sum_{j =1}^{[n-1/2]}  \eta_j^+ \frac{\xi_{2j}(0)+ \xi_{2j+1}(0)}{2} + m_n \dot{x}_n(0) \xi_n(0) \right) \\ 
& - \left( \sum_{j =1}^{[n/2]}  \eta_j^- \frac{\xi_{2j-1}(T)+ \xi_{2j}(T)}{2}  \right)
\end{aligned}
\end{equation}

Four different cases shall be considered below, as we need to use different $\xi$ from $C^{1}([0, T], \rr^n) \cap \Gamma$. Meanwhile we will always assume $\xi$ satisfies
\begin{equation}
\label{eq; xi=0} \sum_{i \in \bn} m_i \xi_i(t) =0, \;\; \forall t \in [0, T].
\end{equation} 
Recall $v_0=\sum_{i \in \bn} m_i \xd_i(t)= 0$ (see Lemma \ref{prop:center-of-mass}), as it will also be needed.

\emph{Case 1:  $n$ is odd and $t=0$.} First let's choose a $\xi$ satisfying
$$ \xi_i(0) = \xi_2(0) \ne 0, \forall i \in \bn \setminus \{1\}; \;  \xi_i(T) =0, \forall i \in \bn.$$ 
By \eqref{eq; xi=0}, $m_1\xi_1(0) = -(m_0 - m_1) \xi_2(0)$. Then \eqref{eq;lim-0-T-odd} implies 
$$ \begin{aligned}
m_1 \xd_1(0) \xi_1(0) &  + \xi_2(0) \sum_{j =1}^{[n-1/2]} \eta_j^+ =0  \stackrel{v_0 =0}{\Longrightarrow} m_1 \xd_1(0) \big( \xi_1(0) - \xi_2(0) \big) =0 \\
&  \Longrightarrow -m_0 \xi_2(0) \xd_1(0) =0 \stackrel{\xi_2(0) \ne 0}{\Longrightarrow} \xd_1(0) =0.
\end{aligned}
$$

Fix an arbitrary integer $j$ with $1 \le j  \le [n/2]$. Then choose a $\xi$ satisfying
$$ \xi_{2j }(0) = \xi_{2j  +1}(0) \ne 0, \; \xi_{i}(0) = 0, \; \forall i \in \bn \setminus \{1, 2j , 2j  +1\}; \; \xi_i(T) =0, \forall i \in \bn. $$
Then \eqref{eq;lim-0-T-odd} and $\dot{x}_1(0)=0$ imply
$$  
   \eta_j^+ \xi_{2j}(0) = 0  \stackrel{\xi_{2j}(0) \ne 0}{\Longrightarrow} \eta_j^+ =0 
$$ 

\emph{Case 2: $n$ is odd and $t=T$.} First choose a $\xi$ satisfying
$$ \xi_i(0) = 0,  \forall i \in \bn; \; \xi_i(T) = \xi_1(T) \ne 0,  \forall i \in \bn \setminus \{n \}. $$ 
By \eqref{eq; xi=0}, $m_n \xi_n(T) = - (m_0 - m_n) \xi_1(T)$. Then \eqref{eq;lim-0-T-odd} implies
$$ \begin{aligned}
 \xi_1(T) \sum_{j=1}^{[n/2]} \eta_j^- & + m_n \xd_n(T) \xi_n(T) =0  \stackrel{v_0 =0}{\Longrightarrow} m_n \xd_n(T) ( \xi_n(T) - \xi_1(T)) =0 \\ 
 & \Longrightarrow - m_0 \xi_1(T) \xd_n(T) =0   \stackrel{\xi_1(T) \ne 0}{\Longrightarrow} \xd_n(T) =0.
\end{aligned}
$$

Fix an arbitrary integer $j$ with $1 \le j \le [n/2]$. Then choose a $\xi \in \Gamma$ with
$$ \xi_i(0) = 0, \forall i \in \bn; \;  \xi_{2j -1}(T) = \xi_{2j}(T) \ne 0,  \xi_i(T) =0,   \forall i \in \bn \setminus \{2j -1, 2j, n \}.$$
Then \eqref{eq;lim-0-T-odd} and  and $\dot{x}_n(T)=0$ imply
$$ \eta_j^- \xi_1(T) =0 \stackrel{\xi_{2j}(T) \ne 0}{\Longrightarrow} \eta_j^- =0. 
$$

\emph{Case 3: $n$ is even and $t=0$.} First choose a $\xi$ satisfying
$$ \xi_1(0) \ne 0, \xi_i(0) = 0, \forall i\in \bn \setminus \{1, n\}; \; \xi_i(T) =0, \forall i \in \bn. $$ 
By \eqref{eq; xi=0}, $m_n \xi_n(0) = - m_1 \xi_1(0)$. Then \eqref{eq;lim-0-T-even} implies
$$  \begin{aligned}
m_1 \xd_1(0) \xi_1(0) + m_n \xd_n(0) \xi_n(0) = 0  & \Longrightarrow m_1 \xi_1(0) \Big( \xd_1(0) - \xd_n(0) \Big) =0  \\ 
& \stackrel{\xi_1(0) \ne 0 }{\Longrightarrow} \xd_1(0) = \xd_n(0).
\end{aligned}
$$ 
Next choose a $\xi$ satisfying
$$\xi_i(0) = \xi_2(0) \ne 0, \forall i \in \bn \setminus \{1, n\}; \; \xi_i(T) =0, \forall i \in \bn. $$
Then \eqref{eq;lim-0-T-even} and $v_0 =0$ imply
$$ \begin{aligned}
m_1 & \xd_1(0) \xi_1(0)  + m_n \xd_n(0) \xi_n(0) - \Big(m_1 \xd_1(0) + m_n \xd_n(0) \Big) \xi_2(0)=0 \\
& \stackrel{\xd_1(0) = \xd_n(0)}{\Longrightarrow} \Big( m_1 \xi_1(0) + m_n \xi_n(0) \Big) \xd_1(0) - (m_1 + m_n) \xi_2(0) \xd_1(0)= 0 \\
& \stackrel{\eqref{eq; xi=0}}\Longrightarrow - m_0\xi_2(0) \xd_1(0)  =0 \stackrel{\xi_2(0) \ne 0}{\Longrightarrow} \xd_1(0) =  \xd_n(0)=0. 
\end{aligned}
$$

Fix an arbitrary integer $j$ with $1 \le j \le \frac{n}{2}-1$. Then choose a $\xi$ with
$$ \xi_{2j}(0) = \xi_{2j+1}(0) \ne 0, \xi_i(0) = 0, \forall i \in \bn \setminus \{1, 2j, 2j+1, n\}; \; \xi_i(T) =0, \forall i \in \bn. 
$$
Using \eqref{eq;lim-0-T-even}, \eqref{eq; xi=0} and $\xd_1(0) =\xd_n(0)=0$, we get
$$ \eta_j^- \xi_{2j}(0) = 0 \stackrel{\xi_{2j}(0) \ne 0}{\Longrightarrow} \eta_j^- =0.$$

\emph{Case 4: $n$ is even and $t=T$.} Fix an arbitrary integer $j$ with $1 \le j \le n/2$ and choose a  $\xi$ satisfying
$$ \xi_i(0) =0,  \forall i \in \bn; \; \xi_{2j -1}(T) = \xi_{2j}(T), \xi_i(T) = \xi_0 \ne 0,  \forall i \in \bn \setminus \{2j-1, 2j\}. $$
By \eqref{eq; xi=0}, 
$$ (m_{2j-1} + m_{2j}) \xi_{2j}(T)= -\xi_0 \sum_{i \in \bn \setminus \{2j-1, 2j \}} m_i.$$  
Combing this with \eqref{eq;lim-0-T-even} and $v_0 =0$, we get
$$  \eta_j^-( \xi_{2j}(T) - \xi_0) = 0  \stackrel{\eqref{eq; xi=0}}\Longrightarrow -\frac{m_0\xi_0}{m_{2j-1} + m_{2j}} \eta_j^-  =0  \stackrel{\xi_0 \ne 0}{\Longrightarrow} \eta_j^- =0.
$$

The rest of property (e) now follows from that fact that $x(t)$ satisfies 
$$ x(t)=x(t+2T)=x(2T -t), \;\forall t \in \rr.$$


\end{proof} 

Let $M=\text{diag}(m_1, \dots, m_1, \dots, m_n, \dots, m_n) \in \rr^{nd \times nd}$ be the mass matrix. Then \eqref{eq:n-body} is equivalent to the following vector field, whose corresponding flow will be denoted as $\phi_t$. 
\begin{equation}
\label{eq;n-body-Hamilton} \dot{y} = v; \;\; \dot{v} = M^{-1}\nabla U(y).  
\end{equation}


\begin{dfn} \label{dfn;sing-block-reg} 
We say a collision singularity $y^* \in \Delta$ is  $C^k$ block-regularizable, if for any collision solution $y^-|_{[t_0-\dl, t_0]}$ ($y(t) \notin \Delta$, $\forall t \ne t_0$) with $y^-(t_0) = y^*$, one can associate a unique ejection solution, $y^+|_{[t_0, t_0 +\dl]}$ ($y(t) \notin \Delta$, $\forall t \ne t_0$) with $y^+(t_0) = y^*$, such that there are two co-dimension one cross sections $\Sigma^{\pm}$ (with $(y^{\pm}, \dot{y}^{\pm})(t_0 \pm \dl) \in \Sigma^{\pm}$ correspondingly) transversal to the vector field \eqref{eq;n-body-Hamilton}, and a $C^k$-homeomorphism $\psi: \Sigma^- \to \Sigma^+$ obtained as an extension of the Poincar\'e map from $\Sigma_-$ to $\Sigma_+$ induced from the flow $\phi_t$ by sending $(y^-, \dot{y}^-)(t_0 -\dl)$ to $(y^+, \dot{y}^+)(t_0 +\dl)$.  
\end{dfn}

\begin{prop} {\cite[Theorem A]{EB93}, \cite[Corollary G]{EB96}} 
A simultaneous binary collision singularity is $C^1$ block-regularizable, when $d=1$ and $C^0$ block-regularizable, when $d =2$ and $3$. 
\end{prop}
\begin{rem}
Block regularization was introduced by Easton \cite{Es71}. When $n=4$, the simultaneous binary collisions are $C^{\frac{8}{3}}$ block-regularizable, when $d =1$ see \cite{MS00} and when $d =2$ see \cite{DD21}. 
\end{rem}


\begin{dfn}
\label{dfn;solution-blok-reg} Given a collision solution $y \in H^1(\rr, \rr^{dn})$ of \eqref{eq:n-body} with the set of collision instants $\{t \in \rr: y(t) \in \Delta \}$ isolated in $\rr$,  we say it is \textbf{a $C^k$ block-regularizable solution} of \eqref{eq:n-body},  if for any collision instant $t_0$, there is a $\dl>0$ small enough, such that $y|_{[t_0, t_0 +\dl]}$ is the unique ejection solution associated with the collision solution $y|_{[t_0 -\dl, t_0]}$ given in Definition \ref{dfn;sing-block-reg} . 
\end{dfn}


 
\begin{proof}[Proof of (f) of Theorem \ref{thm:MainThm}] Since $x(t)$ is $2T$ periodic, it is enough to show $x|_{[t_0 -\dl, t_0+\dl]}$ is block-regularizable, when $t_0 =0$ and $T$. As the argument is almost identical, we will only give the details for $t_0 =0$. 

Since at $t_0=0$, $m_{2j}$ and $m_{2j+1}$ forming the $j$-th binary collision, let  
$$ \xi_j(t) = x_{2j+1}(t) - x_{2j}(t), \; s_j(t) = \xi_j(t)/|\xi_j(t)|, \;\; \forall j = 1, 2, \dots, [(n-1)/2].$$ 
The intrinsic energy of the $j$-th pair then is 
$$ \al_j(t) = \ey \frac{m_{2j}m_{2j+1}}{m_{2j} + m_{2j+1}} |\dot{\xi}_j(t)|^2- \sqrt{ \frac{m_{2j}m_{2j+1}}{m_{2j}+ m_{2j+1}}} \frac{m_{2j}m_{2j+1}}{|\xi_j(t)|}. $$
Despite of the binary collisions at $t=0$, the following limits exist and are finite (see \cite{EB96} or \cite{FT04}), 
\begin{equation*}
 \lim_{t \to 0^{\pm}} s_j(t) = s^{\pm}_j, \; \lim_{t \to 0^{\pm}} \al_j(t) = \al_j^{\pm}.
\end{equation*}
By results from \cite{EB93} and \cite{EB96}, $x|_{[-\dl, \dl]}$ is $C^1$ block-regularizable as a solution of the collinear $n$-body problem, and $C^0$ block-regularizable as a solution of the planar or spatial $n$-body problem, when the following identities hold 
$$ s^-_j = s^+_j \; \text{ and } \; \al_j^- = \al_j^+, \;\; \forall j = 1, 2, \dots, [(n-1)/2]. $$
The first identities follow from the fact that $x(t)$ is collinear with ordering of the mass preserving all the time. Meanwhile by \eqref{eq;symmetry}, $x_i(t) = x_i(-t)$, $\xd_i(t) = -\xd_i(-t)$, $\forall t \in (0, \dl)$ and $\forall i \in \bn$. Hence    
$$ |\xi_j(t)| = |\xi_j(-t)|, \; \dot{\xi}_j(t) = -\dot{\xi}_j(-t), \;\; \forall j =1, 2, \dots, [(n-1)/2].$$
This then implies $\al^-_j = \al^+_j$ for all corresponding $j$'s.  
\end{proof}

\section{Proof of Theorem \ref{thm:extra-symmetry}} \label{sec;extra-sym}
We will only explain it for $\sigma = \text{Identity}$. Notice that in this case, \eqref{eq;mass-condition} becomes 
\begin{equation}
\label{eq;mass-sg-id} m_i = m_{n+1-i}, \;\; \forall i \in \bn. 
\end{equation}

The idea is to impose certain symmetric constraints. Let's recall some basic notations and results following \cite{FT04}. Let $\Lmd= H^1(\rr / 2T, \rr^{nd})$ be the space of $2T$-periodic Sobolev loops. For any finite group $G$,  define its action on $\Lmd$ as follows:
$$ g\big(y(t)\big)= \big(\uprho(g)y_{\usg (g^{-1})(1)}, \uprho(g)y_{\usg (g^{-1})(2)}, \dots, \uprho(g) y_{\usg (g^{-1})(n)}\big) \big(\uptau(g^{-1})t \big), \;\; \forall g \in G, $$
where
\begin{enumerate}
 \item[(a).] $\uprho: G \to O(d)$ represents the action of $G$ on $\rr^d$;
 \item[(b).] $\upsigma: G \to \mathcal{S}_{\bn}$ represents the action of $G$ on $\bn$;
 \item[(c).] $\uptau: G \to O(2)$ represents the action of $G$ on the time circle $\rr/ 2 T \zz$.
\end{enumerate}
Then $\Lmd^G = \{ y \in \Lmd|\; g(y(t)) = y(t), \; \forall g \in G \}$ is the space of \emph{$G$-equivariant loops}. By Palais' symmetric principle, a critical point of $\A$ in $\Lmd^{G}$ is also a critical point of $\A$ in $\Lmd$, when   
\begin{equation}
\label{eq;mass-equal} m_i = m_j, \; \text{ if } \exists g \in G, \text{ such that } \upsigma(g)(i) = j 
\end{equation}

For us, the dihedral group $D_2 : = \langle f, h |\; f^2= h^2 = (fh)^2 =1 \rangle$ with following action will be used
$$  \uprho(f) =\text{Id}, \quad \usg(f) = \text{Id}, \quad \uptau(f) t  =2T- t;$$
$$ \uprho(h) = -\text{Id}, \quad \usg(h) = \prod_{i=1}^{n} (i, n+1-i), \quad \uptau(h) t = \begin{cases}
T -t, & \text{ when } n \text{ is odd}, \\
t, & \text{ when } n \text{ is even}.
\end{cases}$$
For such a group action, \eqref{eq;mass-sg-id} implies \eqref{eq;mass-equal} and each $x \in \Lmd^{D_2}$ satisfies  
\begin{equation*}
\label{eq;action-h-2}  x_i(t) = x_i(2T-t) \; \text{ and } \;   x_i(t) = \begin{cases}
- x_{n+1 -i}(T-t), \; & \text{ when } n \text{ is odd}, \\
- x_{n+1 -i}(t), \; & \text{ when } n \text{ is even}, 
\end{cases} .
\end{equation*}

Let $\Omega=\{ x|_{[0, T]}: x \in \Lmd^{D_2} \}$. Theorem \ref{thm:extra-symmetry} now can be proven by similar arguments as in the proof of Theorem \ref{thm:MainThm} with the desired solutions found as minimizers of $\A$ in $\Omega \cap \Gamma$. 


\hfill\newline
\noindent{\bf Acknowledgement.} 
The author thanks Kuo-Chang Chen and Richard Montgomery for their interests in this paper, and Alain Chenciner for help in reference. He also thanks the anonymous referee for a careful reading of the paper and useful suggestions that helped improving the paper, in particular those on the proof of Property (e) of Theorem \ref{thm:MainThm}.

\bibliographystyle{abbrv}
\bibliography{ref-periodic}

\end{document}